\documentclass[12pt]{amsart}
\pdfoutput=1
\usepackage{amssymb, amsmath}
\usepackage{graphicx} 
\usepackage{textcomp}
\usepackage{blkarray}
\usepackage{easybmat}
\usepackage{multirow}
\usepackage{color}

\newtheorem{theorem}{Theorem}[section]

\newtheorem{lemma}[theorem]{Lemma}
\newtheorem{corollary}[theorem]{Corollary}

\theoremstyle{definition}		

\newtheorem{remark}{Remark}

\title{On links minimizing the tunnel number}
\author{Darlan Gir\~ao, Jo\~ao Miguel Nogueira, Ant\'onio Salgueiro}

\begin{document}
\maketitle

\begin{abstract} 
We show a combinatorial argument in the diagram of large class of links, including satellite and hyperbolic links, where for each of which the tunnel number is the minimum possible, the number of its components minus one.

\end{abstract}

\section{Introduction}\label{intro}

Given a link $L$ in $S^3$, an unknotting tunnel system for $L$ is a collection of disjoint arcs, properly embedded in the exterior of $L$, with the exterior of a regular neighborhood of their union with $L$ being a handlebody. The minimum cardinality of an unknotting tunnel system for $L$ is referred to as the tunnel number of $L$ and is denoted by $t(L)$.
The purpose of this paper is to study the tunnel number of a large class of links and to prove that these links have the lowest possible tunnel number. That is, given a link with $n$ components its tunnel number must be at least $n-1$. To simplify the presentation we denominate these links, referring to them as \textit{band links}.

We define \textit{band links} in the following paragraphs. Consider an embedding $H\times I$ in $S^3$, where $I=[0,1]$ and $H\equiv H\times 0$ is a Heegaard surface of $S^3$.  Let $D_K$ be a $4$-regular graph embedded in $H$, i.e. all vertices have valence 4. Now modify this graph as follows: at each edge of $D_K$ add one or more new vertices. The result is a graph embedded in $H$, which we denote also by $D_K$, whose vertices are either $2$-valent or $4$-valent.

\begin{remark}
	The graph $D_K$ is associated  with the projection of some link $K\subset H\times I$. Conversely, every link $K$ in $H\times I$ can be projected to $H$ so that, after adding additional $2$-valent vertices, its projection is a graph $D_K$ as above.
\end{remark}

Suppose that the projection $D_K$ of $K$ onto $H$ separates $H$ into a collection of disks. From $D_K$ we construct a new graph $D_L$. Consider a small round ball with collateral points $NE, SE, SW$ and $NW$ marked on its boundary. Connect these points by arcs in the interior of the ball, forming one of the following: (a) A \textit{clasp shadow}, which consists of an arc joining $NE$ and $SE$, and an arc joining $NW$ and $SW$. These  arcs have no self-intersections but cross each other twice. (b) A \textit{twist shadow}, which consists of two arcs, connecting each  point in $\{NE,SE\}$ to a point in $\{NW,SW\}$. There may be any number of double-points in a twist shadow, including possibly zero. (c) Now consider a ball with four marked points  $X, Y, Z, W$ cyclically ordered.   A \textit{hash shadow} consists of pair of straight arcs parallel to $XZ$ and a pair of straight arcs parallel to $YW$. These pairs of arcs intersect to  form a square in the interior of the ball.  

\begin{figure}[h]
	\includegraphics[scale=0.15]{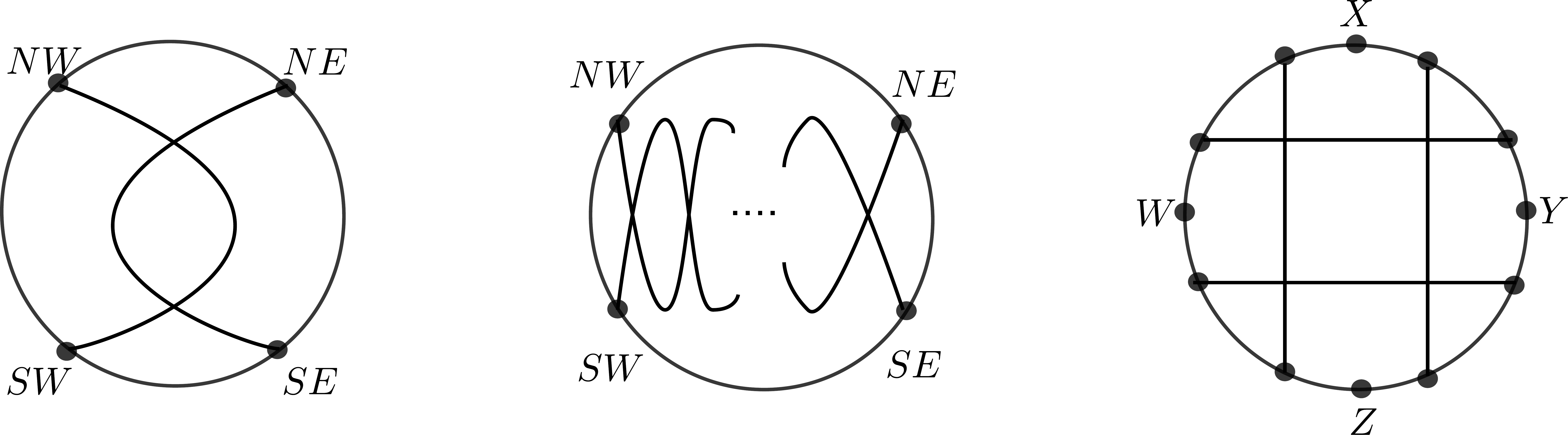}
	\caption{Left: a clasp shadow; Center: a twist shadow; Right: a hash shadow.}
	\label{shadows}
\end{figure}

To build $D_L$, first replace each $4$-valent vertex of $D_K$ by a hash shadow as follows: a small neighborhood (on $H$) of of the vertex is identified to a ball with marked points $X, Y, Z, W$ and arcs $\overline{XZ}, \overline{YW}$. Replace  this neighborhood of the vertex by the corresponding hash shadow.  Next, replace each $2$-valent vertex by a clasp shadow as follows: a small neighborhood of the vertex is a ball with marked cardinal points $W, E$ and an arc $\overline{WE}$. Replace this ball by the corresponding clasp shadow, i.e., cardinal points at the same locations.  At this stage all neighborhoods of vertices have been replaced by either  hash or clasp shadows. The ends of each edge in $D_K$ have been deleted in this process, reducing them to smaller segments. For each such segment, consider a ball (not necessarily round) in $H$ such that the  segment lies inside the ball and its endpoints are the cardinal points $W, E$. Finally, replace each of these neighborhoods of the segments by the corresponding  twist shadow. We note that the number of twists in each twist shadow may be different. 

The last step is to identify the boundary points of the clasp, twist and hash shadows.  These are identified according to the incidence relations of the original graph.

A link $L$ with regular projection $D_L$ onto $H$ is called an \textit{$n$-band link over $K$}, 
where $n$ is the number of connected components of $L$. In Figure \ref{figure:definition4} we have an illustration of this construction of $L$ from $D_K$. 

\begin{figure}[h]
	\includegraphics[scale=0.15]{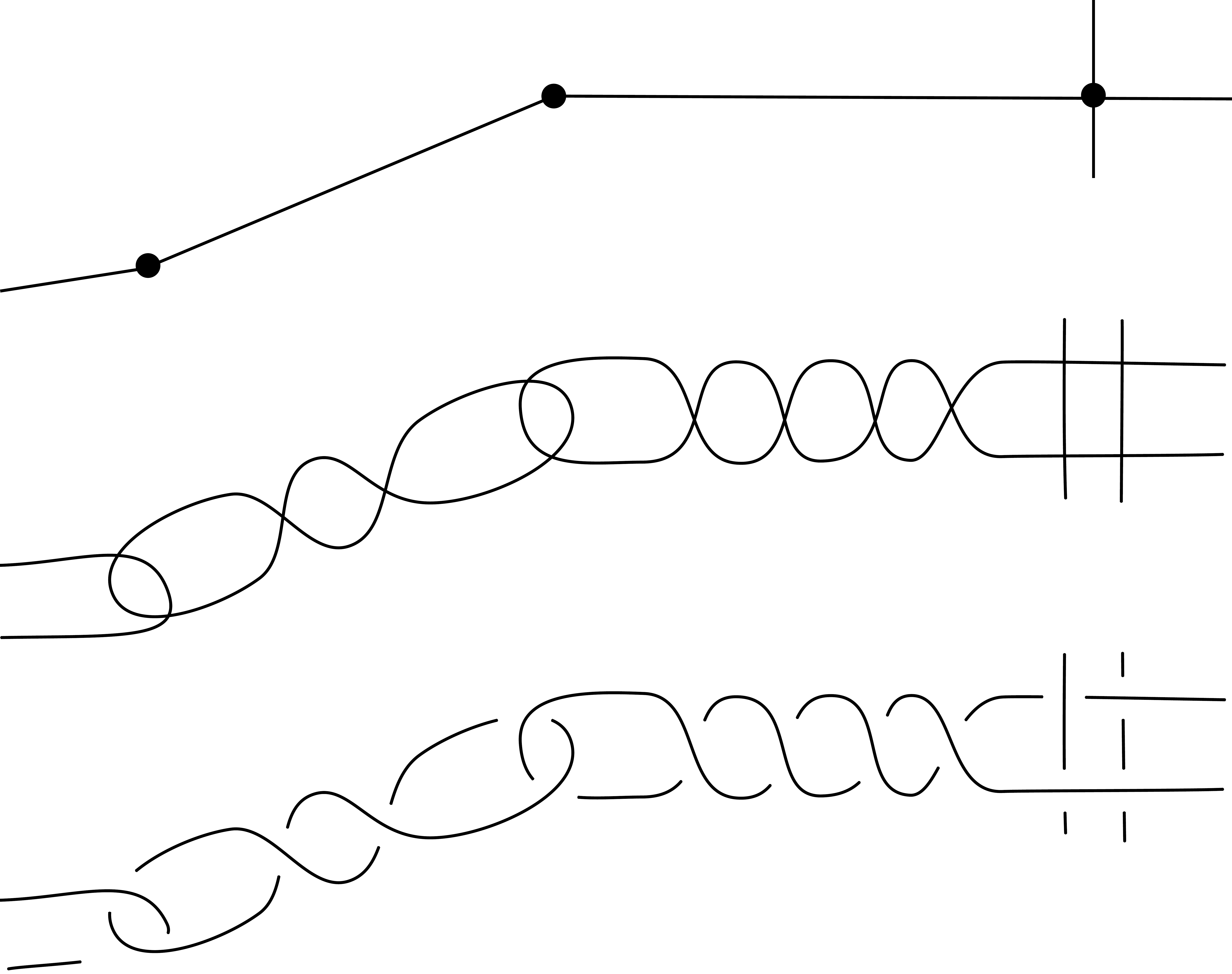}
	\caption{A graph $D_K$, a constructed graph $D_L$ and a corresponding band link $L$.}
	\label{figure:definition4}
\end{figure}

Notice that all components of a band link $L$ are unknotted. Moreover, the projection of each component of $L$ intersects the projections of two other components in two points (corresponding to clasp shadows) each, and the projection of at most one component in four points (corresponding to hash shadows). We observe also that $L$ can be a satellite link, with companion $K$, or a hyperbolic link, for instance, when $D_L$ is alternating on the 2-sphere (from Corollary 2 of \cite{Menasco}).  

In this paper we study the tunnel number of an $n$-band link $L$ over a regular projection $D_K$ of a generic link $K$, and its relation with the rank of the exterior of $L$, denoted $E(L)$. If the projection $D_K$ of $K$ is a simple circle on a sphere, it is straightforward to observe, as in Section \ref{tn}, that the tunnel number of a $n$-band link $L$ over $K$ is $n-1$. In the following theorem we prove that this is also the case for every regular projection $D_L$ of a link $L$.

\begin{theorem}\label{thm1}
	The tunnel number of an $n$-band link exterior is $n-1$.
\end{theorem}

If we consider the Heegaard genus $g(E(L))$ of the exterior of $L$, then it is well known that $g(E(L))=t(L)+1$. Therefore, Theorem \ref{thm1} states that the Heegaard genus of a $n$-band link exterior is $n$.

Waldhausen \cite{Wa} asked whether the rank $r(M)$ of $M$, that is, the minimal number of generators of $\pi_1(M)$, can be realized geometrically as the genus of a Heegaard splitting decomposing $M$ into one handlebody and a compression body, that is if $r(M)=g(M)$, for every compact $3$-manifold $M$. This question came to be known as the Rank versus Genus Conjecture. In \cite{BZ} Boileau--Zieschang provided the first counterexamples by showing there are Seifert manifolds where the rank is strictly smaller than the Heegaard genus. Later Schultens and Weidman \cite{SW} generalized these counterexam\-ples to graph manifolds. Very recently, Li \cite{Li} proved that the conjecture also doesn't hold for all hyperbolic $3$-manifolds. As far as we  know, the conjecture remains open for link exteriors in $S^3$. The first author \cite{Gi} proved this conjecture to be true for augmented links. Theorem \ref{thm1} shows that this is also the case for band links, as stated in the following corollary.   

\begin{corollary}
	If $L$ is an n-band 
	link, then $r\big(E(L)\big)=g\big(E(L)\big)$.
\end{corollary}

In fact, by the  ``half lives, half dies" theorem  (\cite{Ha}, Lemma 3.5) applied to $E(L)$, we have $r\big(E(L)\big)\geq |L|$, where $|L|$ denotes the number of components of $L$. The corollary now follows simply from Theorem \ref{thm1} and the observation that
	$$n=|L|\leq r\big(E(L)\big)\leq  g\big(E(L)\big)=n,$$
 Therefore, $r\big(E(L)\big)=g\big(E(L)\big)=n$. 

\bigskip


This paper is organized as follows. In Section \ref{vt} we describe a procedure to determine an unknotting tunnel system for links from a projection diagram. In Section \ref{combinatorics} we present a combinatorial version of this procedure. Finally, in Section \ref{tn} we use this combinatorial procedure to find the tunnel number of band links. We use the survey \cite{Moriah} by Yoav Moriah as a reference for context on Heegaard decompositions of knot exteriors.

\section*{Acknowledgments}
The first author was partially supported by CNPq grants 446307/\linebreak 2014-9
and 306322/2015-3. He  thanks the hospitality and support from the University of Coimbra during his visit, when part of this work was carried out. The second and third authors were partially supported by the Centre for Mathematics of the University of Coimbra –UID/MAT/00324/2013, funded by the Portuguese Government through FCT/MEC and co-funded by the European Regional Development Fund through the Partnership Agreement
PT2020.
The authors wish to thank Ana Silva and J\'ulio Ara\'ujo  for introducing them to the terminology of section \ref{combinatorics}.  The third author is also thankful for the hospitality while visiting the Federal University of Cear\'a.

\section{Minimal number of vertical tunnels}\label{vt}

Before we proceed we introduce and review some terminology. A \textit{stabilization} of a genus-$g$ Heegaard surface in $S^3$ is a surface of genus $g+1$ obtained by adding a trivial 1-handle, that is, a handle whose core is parallel to the surface. A \textit{destabilization} of a Heegaard surface is a surface obtained from the reverse procedure. Note that a surface obtained by (de)stabilization of a Heegaard surface is also a Heegaard surface.  

Let $K$ be a link in $H\times I\subset S^3$ with a regular projection $D_K$ onto a Heegaard surface $H$ of $S^3$, such that the complement of $D_K$ in $H$ is a collection of disks. Note that we always have such a property when $H$ is a 2-sphere. We refer to these disks, together with their boundary edges and vertices, as \textit{faces}.

We can construct an unknotting tunnel system for $K$ from the crossings of $D_K$ in $H$. In fact, for each crossing $v$ of $D_K$, consider an arc in $H\times I\subset S^3$ connecting the understrand to the overstrand, of the form $v\times I$, as in Figure \ref{figure:vertical}. We call such an arc {\em vertical}. 

\begin{figure}[h]
	\includegraphics[scale=0.12]{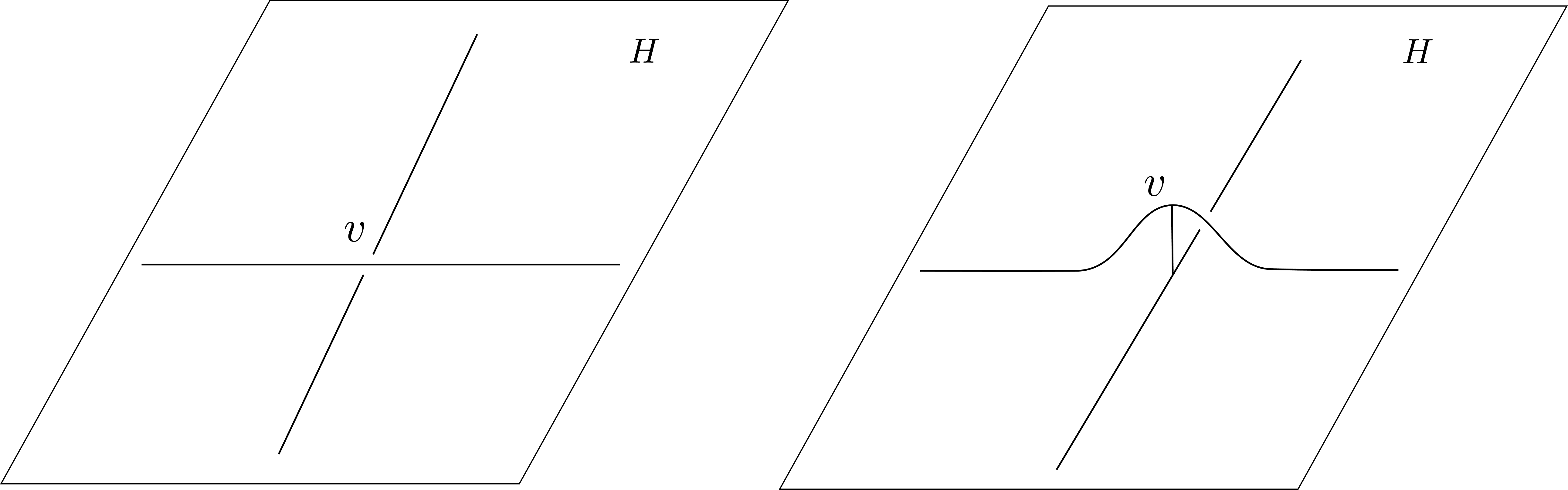}
	\caption{A vertical arc on a crossing $v$ of $D_K$.}
	\label{figure:vertical}
\end{figure}

This collection of vertical arcs, together with $K$ determines a 1-complex homotopically equivalent to $D_K$, which we also denote by $D_K$. As $D_K$ separates the Heegaard surface $H$ into disks, the exterior of $D_K$ in $S^3$  is characterized by two handlebodies connected along 1-handles (with co-cores the disks of $H-D_K$). Hence, the exterior of $D_K$ is a handlebody, and the collection of vertical arcs is an unknotting tunnel system of $H$.





Instead of adding one vertical arc at each crossing, we want to find the minimal number of vertical arcs needed to constitute an unknotting tunnel system. To do this, we will add vertical arcs only at certain crossings and determine if the exterior of the resulting 1-complex is a handlebody, by investigating whether the corresponding decomposition of $E(K)$ is connected by a sequence of (de)stabilizations to the Heegaard decomposition of $E(K)$ obtained by adding one vertical arc at each crossing. 

In this context, we will use the following remark to establish an upper bound for the minimal number of vertical arcs defining an unknotting tunnel system of a band link.

\begin{remark}\label{relationship}   Suppose  one starts to add vertical arcs and, at some point,  there is a face  $f$ determined by the projection $D_K$ 
such that all but one of the crossings of $f$ has a vertical arc. Let $v$ represent this crossing. Then, the vertical arc on the crossing $v$ is trivial with respect to the exterior of $K$ with the vertical arcs added up to this stage, as it is parallel to the edges and vertical arcs with respect to $f$. Hence, if we add the vertical arc at $v$, which we refer to as an \textit{automatic arc}, we have a stabilization of the decomposition defined by $K$ and the vertical arcs added towards the decomposition defined by $D_K$. Therefore, we can add or remove the vertical arc at $v$, without changing whether the decomposition defined by the collection of vertical arcs added is a Heegaard decomposition. 
\end{remark}



\section{Percolation on link projections}\label{combinatorics}

Let $G$ be an embedded graph in a Heegaard surface $H$ with complement in $H$ being a collection of disks, which we refer to as faces, as mentioned before. Consider the following \textit{coloring rule (percolation rule)} on the set $V(G)$, the  vertex set of $G$. 

\subsection{Coloring Rule (Percolation Rule)}\label{rule} 

 Vertices will either be \textit{manually colored} or \textit{automatically colored}.  At each step  $s\in \{0,1,\dots,k,\dots\}$, some subset (possibly   empty) of vertices is manually colored. A vertex $v$ will be  automatically colored at step $s+1$ if it belongs to a face in which all other vertices have already been colored (either manually or automatically) at some previous step. When the last vertex $v$ of $f$ is colored automatically, we also color the face $f$ for notation, as in Figure \ref{figure:removal}.

\begin{figure}[h]
\includegraphics[scale=0.13]{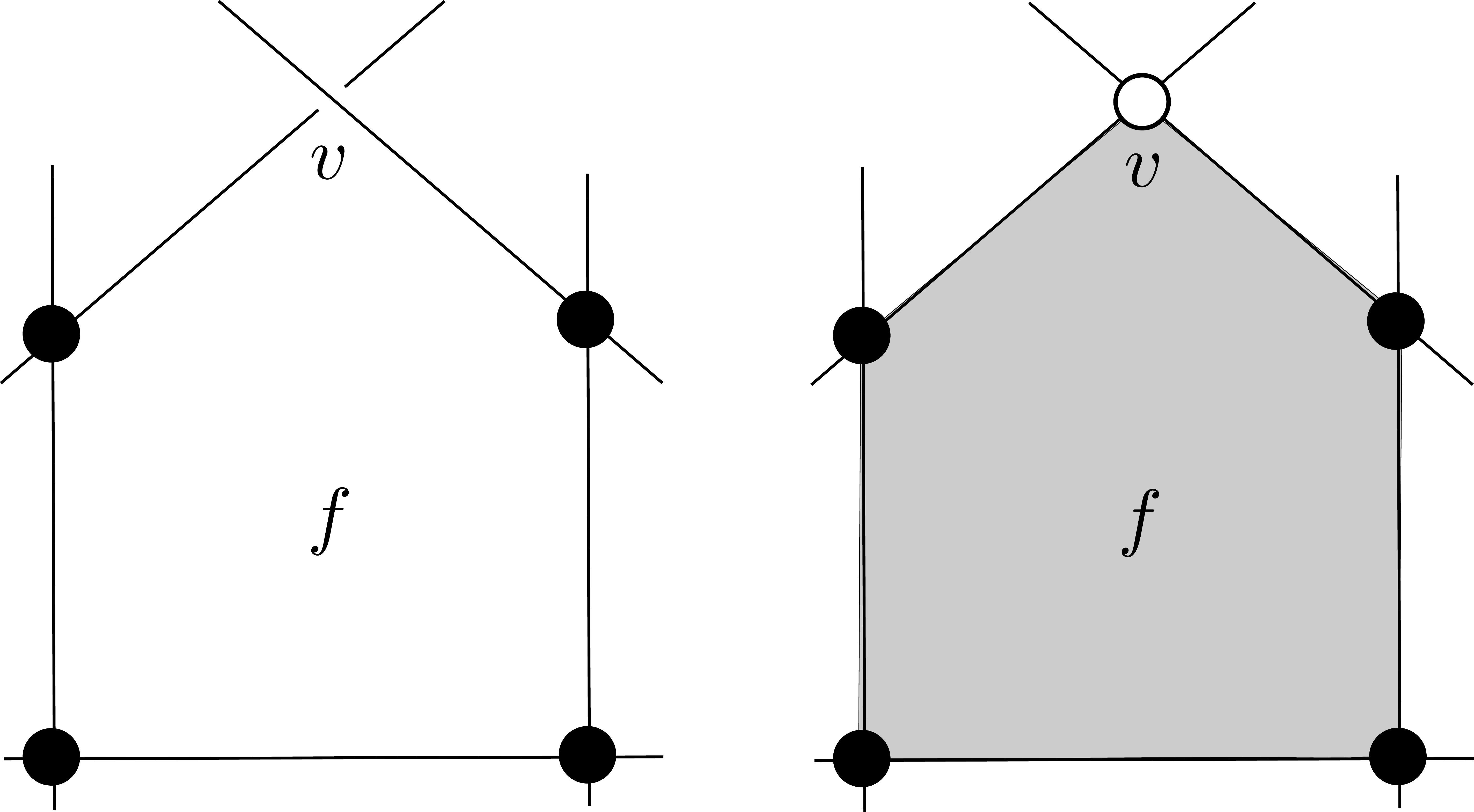}
\caption{Coloring a vertex and a face automatically}
\label{figure:removal}
\end{figure}

Note that manually coloring some vertices at several steps or manually coloring these vertices at step 0 is equivalent, in the sense that every vertex that is automatically colored by one process is automatically colored by the other (not necessarily at the same step).

 We say a subset $V'\subset V(G)$ \textit{percolates} $G$ if manually coloring  all vertices in $V'$ implies all remaining vertices of  $V(G)$ will be automatically colored at some step. A \textit{hull set} for $G$ is a minimal subset $H\subset V(G)$ such that $H$ percolates $G$.  The \textit{hull number} for $G$ is the size of a hull set, denoted $h(G)$.

\subsection{Relationship between hull number and tunnel number}


Let $K$ be a link in $H\times I$ with a regular projection $D_K$ onto a Heegaard surface $H$ of $S^3$, such that the complement of $D_K$ in $H$ is a collection of disks.

\begin{lemma}\label{upper bound}
$t(K)\leq h(D_K)$.
\end{lemma}
\begin{proof}
Let $\mathcal{V}$ be a hull set of $D_K$.\\
At each crossing corresponding to a vertex in $\mathcal{V}$ we add a vertical arc to $K$, and denote this collection of arcs by $\mathcal{V}_a$. We proceed by adding a vertical arc at each crossing corresponding to an automatically colored vertex at each step of the percolation of $\mathcal{V}$. As $\mathcal{V}$ percolates $D_K$, we stop this process only when all crossings of $D_K$ have a vertical arc attached.\\
At each step of the percolation of $\mathcal{V}$, a vertex $w$ is automatically colored because it belongs to a face $f$ with all the other vertices already colored at that step. This translates into having vertical arcs at all crossings of $f$, except at $w$. From Remark \ref{relationship}, adding a vertical arc at $w$ corresponds to a stabilization of the decomposition of $E(K)$ determined by the vertical arcs added up to this step. We add a vertical arc at $w$ and proceed to the next step, until all crossings have a vertical arc added.\\  
This process determines a sequence of stabilizations from the decomposition of $E(K)$ determined by $\mathcal{V}_a$, to the decomposition of $E(K)$ determined by the collection of a vertical arc at each crossing. Following observations from Section \ref{vt}, as the collection of vertical arcs at each crossing determines a Heegaard decomposition, this means that $\mathcal{V}_a$ is an unknotting tunnel system for $K$. 
\end{proof}

\section{Tunnel number of band 
	links}\label{tn}

In this section we prove Theorem \ref{thm1}.
We first prove for a regular projection $D_L$ of a $n$-band link $L$, on some Heegaard surface, that we have $n-1\leq t(L)\leq h(D_L)$. Then, by using the above percolation on these diagrams, we show that $h(D_L)\leq n-1$,
and Theorem \ref{thm1} immediately follows from the inequalities $n-1\leq t(L)\leq h(D_L)\leq n-1$. 

%
%
%


As a warm-up, we look at the case when $L$ is a band over the unknot. Here, it is fairly easy to see that $t(L)=n-1$. In fact, if the  components of $L$  are cyclically labeled $C_1, C_2,\ldots, C_n$, adding  one vertical arc between  all pairs of consecutive components $C_i, C_{i+1}$, for $i=1,\ldots,n-1$, yields an unknotting tunnel system for $L$. This procedure is illustrated in Figure \ref{tunnel_system}.

\begin{figure}[h]
	\includegraphics[scale=.075]{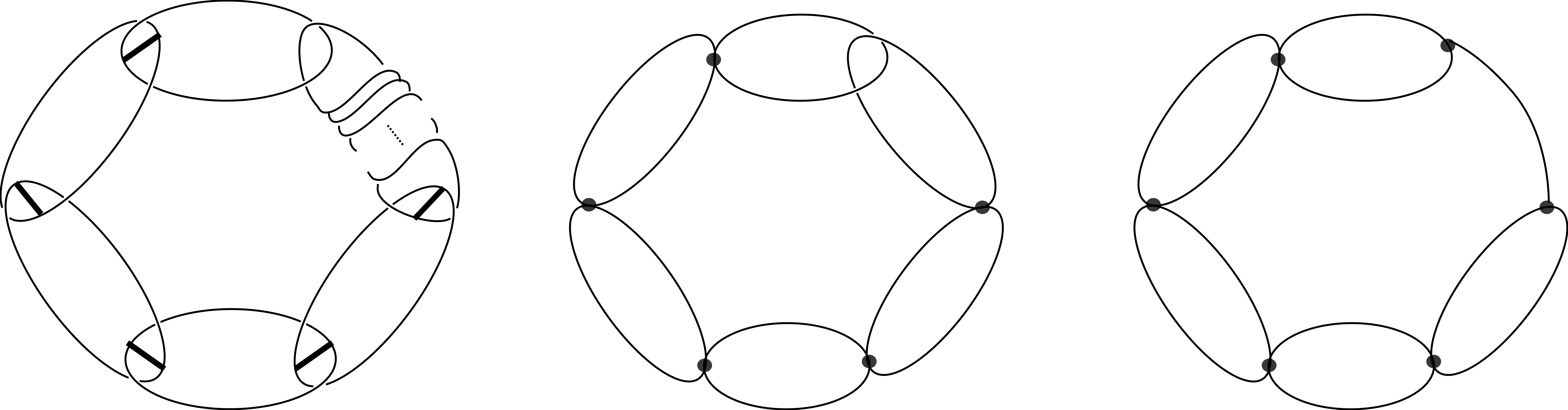}
	\caption{Left:1-complex $K'$ (tunnels are the bold segments); Middle: collapse tunnels to vertices; Right: $K'$ can be made planar  after a sequence of handle slides.}
	\label{tunnel_system}
\end{figure}

If we attempt this procedure for a generic band link $L$, by adding one vertical arc between all pairs of consecutive components, we see that the resulting 1-complex has a similar diagram to the knot $K$. What this means is that, following this approach, we would need to add another $t(K)$ arcs to this 1-complex in order to obtain an unknotting tunnel system for $L$. Instead, we will follow a different strategy to realize that adding $n-1$ vertical arcs suffices to obtain an unknotting tunnel system for a $n$-band links, which is the minimum possible. This is achieved by the above percolation procedure applied to $D_L$. First we make the following remark.


\begin{remark}\label{minimum hull} 
Let $D_L$ be a projection of an $n$-band link. Then $h(D_L)\geq n-1$. 
In fact, we  know that the tunnel number of a link $L$ is at least its number of components minus 1. Combining this with Lemma \ref{upper bound}, we obtain $n-1\leq t(L)\leq h(D_L)$. 
\end{remark}

We are now ready to present the proof of Theorem \ref{thm1}.

\begin{proof}[Proof of Theorem \ref{thm1}]
Let $L$ be an $n$-band link with corresponding projection $D_L$, as in the definition of a band link. We will show that $h(D_L)\leq n-1$. The theorem then follows by considering Remark \ref{minimum hull}.

First we will deal with the case in which all twist shadows have no double-points. 
In this situation, the projection of each component of $L$ is a simple closed curve and intersections in the projection are as in Figure \ref{crossings}. We refer to these simple closed curves as \textit{circle components}.

\begin{figure}[h]
\includegraphics[scale=.15]{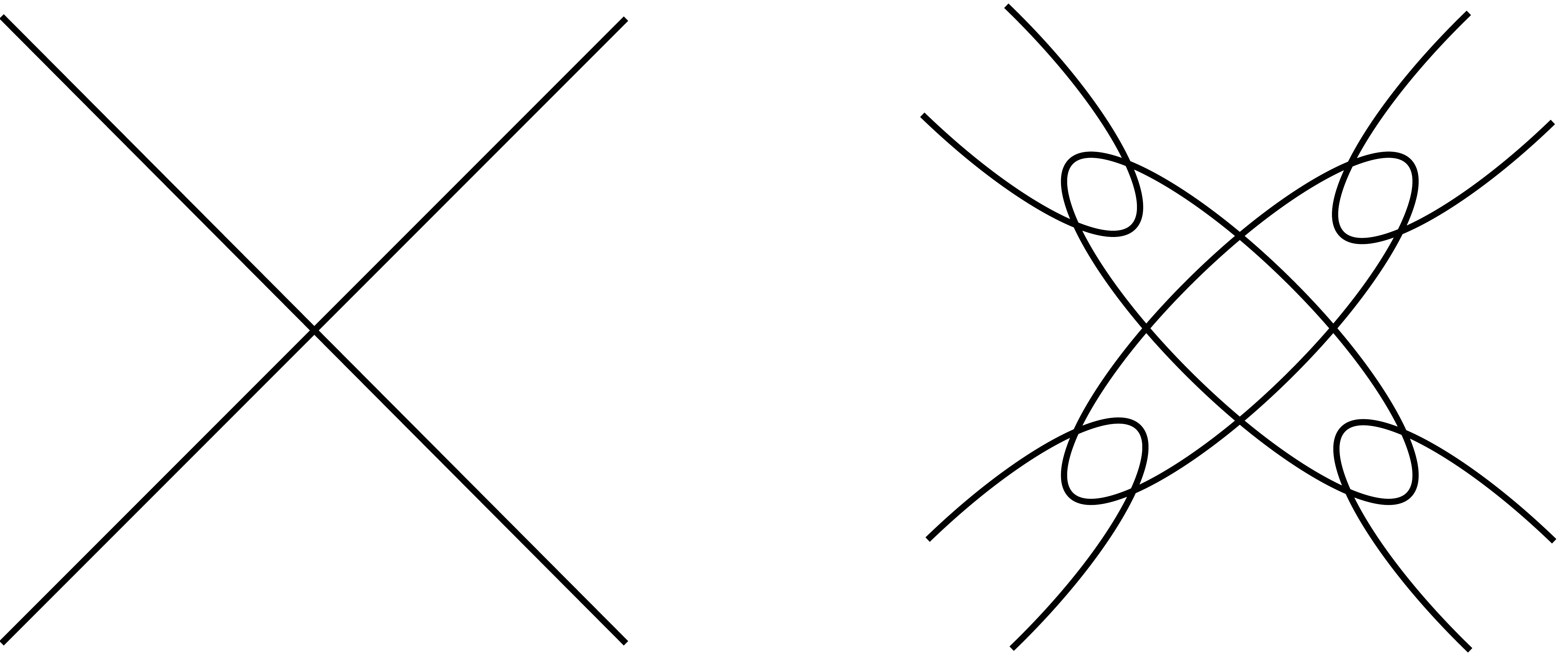}
\caption{Left: a crossing of $D_K$; 
	Right: portion of the band 
	 corresponding to this crossing in $D_L$.}
\label{crossings}
\end{figure}

There are  two types of faces in the projection $D_L$: 
 those which arise from faces of the projection $D_K$ and those which are within a projection of a circle component of $L$.

\medskip

\textbf{Step 1.} Choose a face $f_0$ of $D_L$  
corresponding to a face of the projection $D_K$.

In the face $f_0$, manually color all vertices except one. By the coloring rule, the remaining vertex is automatically colored. Notice that if the face $f_0$ has $m$ vertices, then there are $m$ circle components making up $f_0$.  

Since the projection of every component of $L$ intersects at most one other component in four crossings, 
then, 
by allowing further automatic  colorings, we observe that all faces sharing a vertex with $f_0$  have their remaining vertices automatically colored. As mentioned in Section \ref{rule}, we also color all these faces having all of their vertices colored, and denote this region by $R_0$. These steps are illustrated in  Figure \ref{face}. 

\begin{figure}[h]
	\includegraphics[scale=.09]{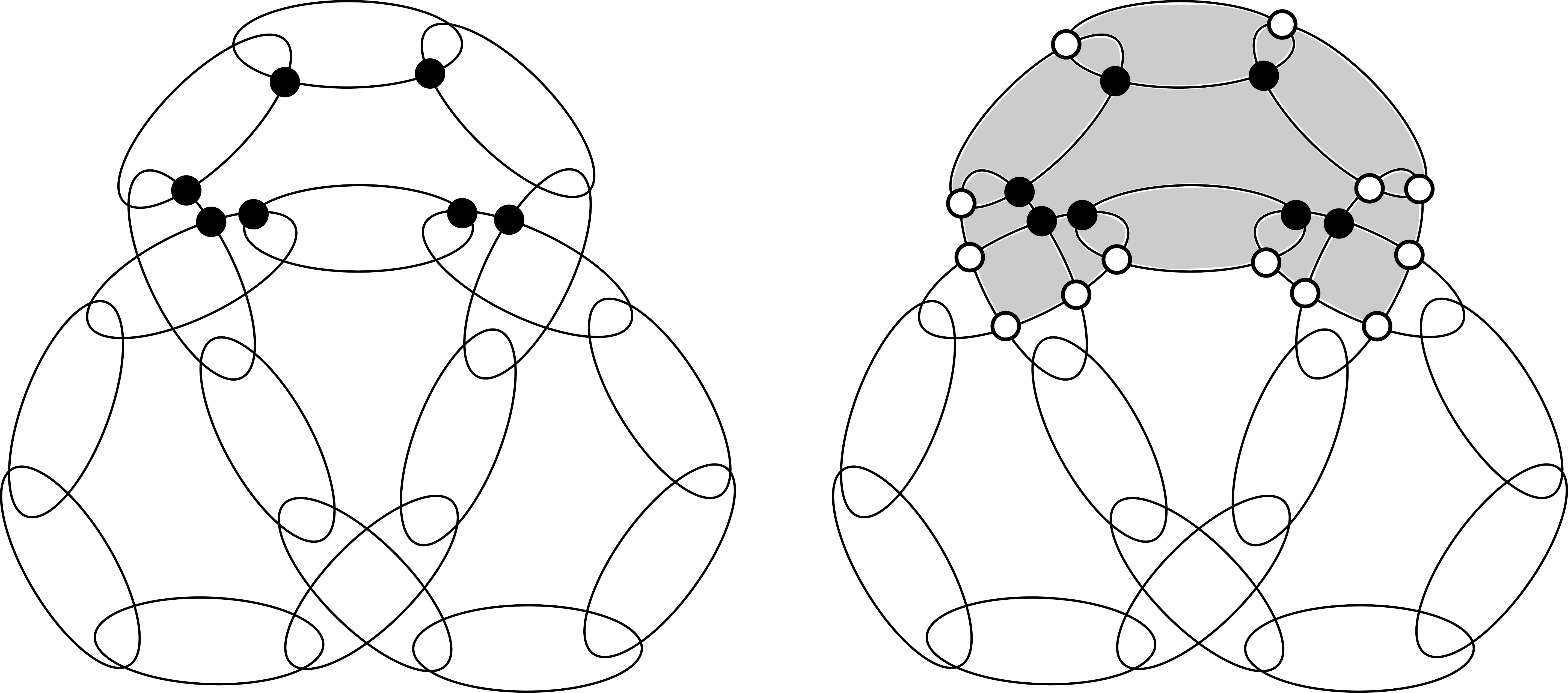}
	\caption{Left: black vertices are colored; Right: white vertices are colored automatically and faces sharing a vertex with $f_0$ are colored.}
	\label{face}
\end{figure}

\medskip

\textbf{Inductive Step.} Consider a  face  $f_1$ of $D_L$, corresponding to a face of $D_K$, 
adjacent to $R_0$, and innermost with respect to $R_0$ (that is, its intersection with $R_0$ is connected). If there is no such face then $R_0$ is the entire surface $H$, and all vertices have been colored. Otherwise, the face $f_1$ has $m_1$ vertices, other than the ones in $R_0$. Again, since the projection of every component of $L$ intersects at most one other component in four crossings, 
then these vertices are given by intersections of  $m_1-1$ circle components in which no vertices have been colored yet. Manually color $m_1-1$ vertices of $f_1$.  Since all vertices in $R_0$ are already colored, only one vertex of $f_1$ remains to be colored, and thus it is automatically colored. Color the face $f_1$, together with all faces within a projection of a circle component and which share a vertex with $f_1$. We denote the new colored region by $R_1$, which is still connected as $f_1$ is adjacent to $R_0$.  

\medskip

In general, we consider a face $f_i$ adjacent to $R_{i-1}$, and innermost with respect to $R_i$, with $m_i$ vertices, given by intersections of $m_i-1$ circles, and call the resulting colored region $R_i$, for $i\geq 1$. If there is no such face $f_i$ adjacent to $R_{i+1}$, then $R_{i-1}$ is the entire surface $H$, and all vertices have been colored.

Eventually this process must terminate, and all faces of $D_L$ will be colored, that is all vertices of $D_L$ will be colored, either automatically or manually. With the above steps we determined a subset of vertices of $D_L$, the ones which were manually colored, that percolates $D_L$. Observe also that we manually colored exactly $n-1$ vertices. In fact, at Step 1 we colored $m-1$ vertices of the $m$ vertices of $f_0$, and in the face $f_j$, for $j\geq 1$, we manually color a single vertex for each circle component with no colored vertices. Note that if all circle components adjacent to $f_j$ already have a colored vertex, as $f_j$ is innermost with respect to $R_j$, there is only one vertex left in $f_j$ to color, which then is automatically colored. Therefore, $h(D_L)\leq n-1$. This proves the theorem in the case where all arcs of $D_L$ are untwisted. 

For the general case, we just need to observe that all additional vertices coming from twisting will be colored automatically when we color one crossing between two different components.
\end{proof}




\section{Further remarks}
The method of percolation on link projections to estimate the tunnel number behaves well for band links. Although this method may not find optimal bounds in general, it seems to be useful in other classes of links. For example, in \cite{Gi} the tunnel number of \textit{augmented links} has been determined. The method of percolation can, after appropriate choice of vertices, be used to obtain the the results therein. 

A particular subclass of band links is that of \textit{chain links}, which can be  thought as the links obtained from a real chain, i.e., the link given by a sequence of trivial circles connected end-to-end. The first link of the chain is then connected  to the last, along a knot $K$. 	One of our initial goals was to test the percolation method on diagrams of chain links, but it quickly made more sense to work on the broader class. Diagrammatically, chain links are almost as complicated as band links, but the latter is a much larger class. For instance, chain links are not hyperbolic when defined along a nontrivial knot $K$, but many band links are.   

Additionally we would like to emphasize that the class of band links might be interesting to be studied separately. For example, one could ask which geometric, topological or combinatorial  properties of a band link $L$ can be predicted from the graph $D_K$ from which it was built. Or one could ask how broad this class is, in some suitable sense.

\vspace{.4cm}
\noindent
\address{\textsc{Department of Mathematics,\\
Universidade Federal do Cear\'a}}\\
\email{\textit{E-mail:}\texttt{ dgirao@mat.ufc.br}}

\vspace{.4cm}
\noindent
\address{\textsc{Department of Mathematics,\\
University of Coimbra}}\\
\email{\textit{E-mail:}\texttt{
		nogueira@mat.uc.pt}}

\vspace{.4cm}
\noindent
\address{\textsc{Department of Mathematics,\\
		University of Coimbra}}\\
\email{\textit{E-mail:}\texttt{
		ams@mat.uc.pt}}


\begin{thebibliography}{99}

\bibitem[1]{BZ} M. Boileau, H. Zieschang, \textit{Heegaard genus of closed orientable Seifert 3-manifolds}, Invent. Math. 76 n. 3 (1984), pp. 455-468.
\bibitem[2]{Gi} D. Gir\~ao, \textit{Heegaard genus and rank of augmented link complements}, Math. Z.,  Vol. 281 (2015), Issue 3, pp. 775-782.
\bibitem[3]{Ha}  A. Hatcher, \textit{Algebraic Topology}, Cambridge University Press, 2001. 
\bibitem[4]{Li}T. Li, \textit{Rank and genus of 3-manifolds},  J. Amer. Math. Soc. 26 (2013), pp. 777-829.
\bibitem[5]{Menasco} W. Menasco, Closed incompressible surfaces in alternating knot and link complements,
Topology 23 n. 1 (1984), pp. 37-44.
\bibitem[6]{Moriah} Y. Moriah, Heegaard splittings of knot exteriors,
Geometry \& Topology Monographs 12 (2007), pp. 191-232.
\bibitem[7]{SW} J. Schultens, R. Weidman, \textit{On the geometric and the algebraic rank
of graph manifolds}, Pacific J. Math. 231 (2007), pp. 481-510.
\bibitem[8]{Wa} F. Waldhausen, \textit{Some problems on 3-manifolds}, Proc. Symposia in Pure Math. 32, (1978), pp. 313-322.

 

\end{thebibliography}
\end{document}